\documentclass[11pt]{article}

\usepackage{xcolor} 

\usepackage{enumerate}

\usepackage{hyperref}

\usepackage{amsmath}
\usepackage{amsthm}
\usepackage{amssymb}
\usepackage{amsfonts}
\usepackage{amsrefs}
\usepackage{bbm}
\usepackage[normalem]{ulem}

\usepackage[title, toc]{appendix}

\usepackage{tikz}
\usetikzlibrary{arrows}
\usetikzlibrary{decorations.markings}
\usepackage{color}
\usepackage{graphicx}
\usepackage{fullpage}
\usepackage{float}
\usepackage{wrapfig}
\usepackage{caption}
\usepackage{subcaption}
\usepackage{listings}
\usepackage{verbatim}
\numberwithin{equation}{section}

\DeclareMathOperator{\Var}{Var}
\DeclareMathOperator*{\diag}{diag}

\DeclareMathOperator*{\rank}{rank}

\newcommand{\beq}{ \begin{equation} }
\newcommand{\eeq}{ \end{equation} }
\newcommand{\beqq}{ \begin{equation*} }
\newcommand{\eeqq}{ \end{equation*} }

\newcommand{\ip}[1]{\langle {#1} \rangle }
\newcommand{\ipa}[1]{\left\langle {#1} \right\rangle }

\def \dd {\mathrm{d}}

\def \a {\alpha}
\def \b {\beta}

\def \d {\delta}
\def \la {\lambda}
\def \s {\sigma}

\def \r {\rho}

\newcommand{\bR}{\mathbb{R}}

\newcommand{\bE}{\mathbb{E}}

\def \EE {\mathbb{E}}

\def \RR {\mathbb{R}}

\def \cN {\mathcal{N}}

\def \cR {\mathcal{R}}

\newtheorem{theorem}{Theorem}[section]

\newtheorem{lemma}[theorem]{Lemma}

\newtheorem{definition}[theorem]{Definition}

\theoremstyle{remark}

\begin{document}

\title{Order of fluctuations of the free energy in the positive semi-definite MSK model at critical temperature}
\author{Elizabeth W. Collins-Woodfin\footnote{Department of Mathematics \& Statistics, McGill University,
Montreal, QC, H3A 0G4, Canada \newline email: \texttt{elizabeth.collins-woodfin@mail.mcgill.ca}} \and Han Gia Le\footnote{Department of Mathematics, University of Michigan,
Ann Arbor, MI, 48109, USA \newline email: \texttt{hanle@umich.edu}}}
	\date{\today}

	\maketitle

\begin{abstract}
In this note, we consider the multi-species Sherrington-Kirkpatrick spin glass model at its conjectured critical temperature, and we show that, when the variance profile matrix $\Delta^2$ is positive semi-definite, the variance of the free energy is $O(\log^2N)$.  Furthermore, when one approaches this temperature threshold from the low temperature side at a rate of $O(N^{-\a})$ with $\a>0$, the variance is $O(\log^2N+N^{1-\a})$.
This result is a direct extension of the work of Chen and Lam (2019) who proved an analogous result for the SK model, and our proof methods are adapted from theirs.
\end{abstract}

\section{Introduction}
The Sherrington--Kirkpatrick (SK) model is a model in statistical mechanics, introduced in 1975 to study magnetic alloys \cite{SherringtonKirkpatrick1975}. Its Hamiltonian $H_N^{SK}: \{-1, 1\}^N \to \bR$ is given by $\s \mapsto \frac1{\sqrt N}\sum_{i,j=1}^N g_{ij}\s_i\s_j$, where $(g_{ij})_{i,j=1}^N$ are independent standard Gaussian random variables. The free energy of the model at each inverse temperature $\b>0$ is defined as 
\beqq
F^{SK}_N(\b) = \log\sum_{\s \in \{-1,1\}^N} \exp\left(\b H^{SK}_N(\s)\right)
\eeqq
As $N\to\infty$, the limiting free energy is known for every $\b>0$, established in breakthrough works \cite{Parisi1979, Parisi1980} and \cite{Guerra03, TalagrandSK, Panchenk14}. The fluctuations are also known in the high temperature regime $\b<\frac1{\sqrt2}$ \cite{AizenmanLebowitzRuelle}. The fluctuations remain open for the low temperature regime. 

There is great interest in obtaining the order of fluctuations of the free energy at the critical temperature $\b_c:=\frac1{\sqrt2}$. In particular, the works \cite{Talagrandvol1, Talagrandvol2, Chatterjee09} together yield $\Var(F_N^{SK}(\b_c))= O(N^{1/2})$. At the time of this note, the best bound is due to Chen and Lam \cite{ChenLam19}, where they successfully showed that 
\beq
\Var(F_N^{SK}(\b_c))= O(\log^2 N ).
\eeq
This is a step closer to the conjecture $\Var(F_N^{SK}(\b_c))= \frac16 \log N + O(N^{-2})$, based on various works including \cite{conjSk2,conjSK1, ParisiRizzo09}.

In this note, we study the variance bound of the free energy for multi-species SK (MSK) model at the critical temperature, with the goal of deriving a bound analogous to the one in \cite{ChenLam19} for the SK model.

In contrast to the SK model, the spins in the MSK model are divided into a fixed number of species, and the proportion of each species is asymptotically fixed as the number of spins goes to infinity. Moreover, the variance of the spin interactions $g_{ij}$ is no longer constant (e.g., it is constant 1 in the previously described SK model), but depends on the species structure.
 
\begin{definition}[MSK model]
    Fix $k\geq 1$. Let $\Lambda = \diag(\a_1,\dots, \a_k)$ be a matrix satisfying $\a_i\in (0,1)$ for all $i$ and $\sum_{i=1}^k \a_i=1$. Let $\Delta^2$ be a $k\times k$ symmetric matrix of non-negative entries. 
    An MSK model with species density matrix $\Lambda$ and variance profile matrix $\Delta^2$ satisfies the following assumptions.
    \begin{itemize}
        \item Given $N \geq k$, there is a partition of spins $\{1,2,\dots, N\} = \cup_{s=1}^k I_s$, where
        $\Lambda_N:=\frac1N\diag(|I_1|,...,|I_k|)$ and
         $\lim_{N\to\infty}\Lambda_N=\Lambda$.
        \item Given a spin configuration $\s\in \{-1,1\}^N$, the Hamiltonian is given by 
        \beq
        H_N(\s) = \frac1{\sqrt N}\sum_{i,j=1}^N g_{ij}\s_i\s_j
        \eeq
        where $g_{ij}$ are independent centered Gaussian variables and $\bE[g_{ij}^2] = \Delta^2_{s,t}$ for $i\in I_s$ and $j\in I_t$.
    \end{itemize}
\end{definition}

We consider the free energy of the model at inverse temperature $\b>0$, which is defined as 
\beq\label{eq:fe}
F_N(\b)= \log \sum_{\s \in \{-1,1\}^N} \exp \left(\b H_N(\s)\right).
\eeq

In the case where $\Delta^2$ is positive-definite, the limiting free energy is known for all fixed $
\b>0$. In particular, Barra et al. in  \cite{Barra15} proposed Parisi’s formula for the free energy and proved the upper bound. The matching lower bound is verified by Panchenko \cite{Panchenko15}. Notably, the lower bound does rely on the positive-definite assumption. 

In addition, fluctuations of the free energy in the high temperature regime are known. Here, our notion of high and low temperature regimes is analogous to those for the SK model in the sense that the high temperature regime refers to the subset of the parameter space where the so-called replica symmetric solution holds, while in the low temperature regime, there is a ``replica symmetry breaking" (RSB) phenomenon. For the SK model, the so-called de Almeida--Thouless (AT) line is conjectured to be the boundary between these two regimes. See 
\cite{Talagrandvol1, Talagrandvol2}
for more detailed discussions. Regarding the MSK model, the authors in \cite{BatesSolmanSohn19} derived the analogous AT line condition for a two-species SK model, with a closed-form formula for the temperature threshold $\b_c$ in terms of $\Lambda$ and $\Delta^2$. This result is extended to $k$-species SK, for $k\geq 2$, in \cite{DeyWu21}. 
In the case of no external field (which is the setting of this note) the authors of \cite{DeyWu21} prove that MSK is replica-symmetric for $\b<\b_c$ where
\beq\label{eq:bcrit}
\b_c= \rho(2\Lambda \Delta^2)^{-1/2}
\eeq
and $\rho(\cdot)$ denotes the spectral radius.  They demonstrate examples of MSK model that are RSB for all $\b>\b_c$, but the proof for general MSK remains open (hence $\b_c$ is the conjectured critical inverse temperature). The authors in \cite{DeyWu21} then derive a law of large numbers and a central limit theorem for the free energy at every $\b< \b_c$, even when $\Delta^2$ is indefinite (the result only holds for $\b<\b_0<\b_c$ where there exists an external field, however). See, for example, \cite{Albericietal21, DeyWu23, Wu24} for discussions of RSB and other results in MSK model, some of which do not rely on the positive semi-definite assumption.

In this note, we verify that $\Var(F_N(\b_c))= O(\log^2 N)$ under the assumption that $\Delta^2$ is positive semi-definite. In addition, under the same assumption, we obtain an upper bound on the fluctuations when we approach the critical temperature from the low temperature regime.
\begin{theorem}\label{thm:maintheorem}
    Consider an MSK model with species density matrix $\Lambda=\Lambda_N+O(N^{-1})$ and variance profile matrix $\Delta^2$. Let $F_N(\b)$ be the free energy as defined in \eqref{eq:fe}, and $\b_c$ be as given in \eqref{eq:bcrit}. If $\Delta^2$ is positive semi-definite, then the following statements hold:
    \begin{enumerate}
    \item There exists a constant $C>0$ such that 
    \beq
    	\Var\left(F_N(\b_c)\right) \leq C \left( (\log N)^2 +1\right).
    \eeq
    \item For fixed $\a>0$ and $d>0$, there is a constant $C>0$, depending only on $\a$ and $\d$, such that
    \beq
    	\Var\left(F_N(\sqrt{\b_c^2+dN^{-\a}})\right) \leq C \left( (\log N)^2 + N^{1-\a}\right).
    \eeq
    \end{enumerate}
\end{theorem}

\subsection{Proof set-up}

Understanding the variance of the free energy begins with the basic observation that, for spin vectors $\s,\r$, the Hamiltonian has covariance
\beq
\EE H_N(\s)H_N(\r)=N\cR(\s,\r):=Nv^T\Delta^2 v\quad \text{where }v=\begin{bmatrix}R_1\\ \vdots \\ R_k\end{bmatrix}\text{ and }R_s=\frac1N\sum_{i\in I_s}\s_i\r_i.
\eeq
We refer to $\{R_s\}_{s=1}^k$ as the within-species overlaps and $\cR$ as the multi-overlap.  

The starting point of the proof of our theorem is a variance identity, obtained by Chatterjee \cite{Chatterjee09} in a study of disorder chaos in spin glass.  In particular, the variance of the free energy can be expressed in terms of an interpolating Gibbs measure.  Let $H'_N,H''_N$ be independent copies of $H_N$ and, for $t\in[0,1]$ and $\s,\r\in\{-1,1\}^N$, define
\beq\begin{split}
H^1_{N,t}(\s)&=\sqrt{t}H_N(\s)+\sqrt{1-t}H'_N(\s),\\
H^2_{N,t}(\r)&=\sqrt{t}H_N(\r)+\sqrt{1-t}H''_N(\r).
\end{split}\eeq
Then we have covariance $\EE H_{N,t}^1(\s)H_{N,t}^2(\r)=tN\cR(\s,\r)$.
Applying Theorem 3.8 from \cite{Chatterjee09} (with the change of variable $e^{-2t}\mapsto t$) yields the identity
\beq\label{eq:Chatterjee}
\Var(F_N(\b))=\b^2N\int_0^1\EE\ip{\cR(\s,\r)}_t\dd t,
\eeq
where $\ip{\cdot}_t$ denotes Gibbs expectation with respect to the interpolating measure
\beq
\frac{\exp(\b(H^1_{N,t}(\s)+H^2_{N,t}(\r)))}{\sum_{\s,\r}\exp(\b(H^1_{N,t}(\s)+H^2_{N,t}(\r)))}.
\eeq

Using this identity, the main work in proving our theorem lies in bounding the quantity $\EE\ip{\cR(\s,\r)}_t$.  This requires a second interpolation argument, adapted from \cite{ChenLam19}, which we present in the next section.


\subsection{Beyond the positive semi-definite case}

It is natural to wonder if a similar bound applies to the case of indefinite $\Delta^2$. For instance, the bipartite SK model, with $k=2$ and $\Delta^2=[\begin{smallmatrix}0&1\\1&0\end{smallmatrix}]$, is a special example.  Our restriction to the case of positive semi-definite $\Delta^2$ arises in the proof of Lemma \ref{lem:mainlemma}.  In particular, we use the fact that, when $\Delta^2$ is positive semi-definite, $\cR(\s,\r)\geq0$ for all $\s,\r$ (the SK model corresponds to the case of $k=1$, where this inequality holds trivially).  Furthermore, in the proof for Lemma \ref{lem:Talagrand_bound}, we use the fact that there exists some multivariate Gaussian distribution with $\Delta^2$ as its covariance matrix.  Again, this holds trivially in the $k=1$ case (SK), but requires positive semi-definite $\Delta^2$ in general.  Modifying these proofs (or developing entirely new ones) to treat the case of indefinite $\Delta^2$ remains a goal for future work. 

\section{Proof}
Using the variance identity \eqref{eq:Chatterjee}, our main task is to bound the double expectation $\bE \langle \cR(\s, \r) \rangle_{t}$, which is achieved in the following lemma.

\begin{lemma}\label{lem:mainlemma} Consider a $k$-species model with positive semi-definite $\Delta^2$ and let $r=\rank(\Delta^2)$. Then, for any $t\in[0,1]$ such that $\b^2\b_c^{-2}t<1$, we have 
\beq 
\EE\ip{\cR(\s,\r)}_{t}\leq\frac{r}{N}\frac{\b_c^{-2}}{1-\b^2\b_c^{-2}t}\log\frac{2}{1-\b^2\b_c^{-2}t}.
\eeq
\end{lemma}
This lemma and its proof are adapted from Lemma 2.1 of \cite{ChenLam19}.  Indeed their lemma is a special of case of this one, and our methods are an extension of theirs.
We begin by showing that Lemma \ref{lem:mainlemma} implies Theorem \ref{thm:maintheorem} and then we prove Lemma \ref{lem:mainlemma}.

\begin{proof}[Proof of Theorem \ref{thm:maintheorem}]

Let $\b\geq \b_c$. Let $0<\delta<\b_c^{2}\b^{-2} \leq 1$, where $\delta$ is a constant to be chosen. We apply Lemma \ref{lem:mainlemma} to the variance identity \eqref{eq:Chatterjee} to obtain 
\beqq
\begin{split}
\Var (F_N(\b)) 
&= \b^2N \int_0^1 \bE \langle \cR(\s, \r) \rangle_{t}\dd t \\
&= \b^2N \int_0^{\b_c^2\b^{-2}-\delta} \bE \langle \cR(\s, \r) \rangle_{t}\dd t + \b^2N \int_{\b_c^2\b^{-2}-\delta}^1 \bE \langle \cR(\s, \r) \rangle_{t} \dd t \\
&\leq  \int_0^{\b_c^2\b^{-2}-\delta} \frac{r\b^2\b_c^{-2}}{1 - \b^2\b_c^{-2}t}\log\frac{2}{1 - \b^2\b_c^{-2}t}\dd t + C\b^2 N \left(1-\b_c^2\b^{-2} + \delta\right). 
\end{split}
\eeqq
Here, $C:=\max_{\s, \r\in \{-1,1\}^N}\cR(\s,\r) = \sum_{s,t=1}^k\Delta^2_{st}\a_s\a_t $ is a constant, where $\diag(\a_1,\dots, \a_k)$ is the species density matrix. Moreover, the last integral is equal to 
\beq
\int_{\b^2\b_c^{-2}\delta}^1 \frac ru\log \frac2u \dd u 
= -\frac{r}{2}\left.\left(\log \frac{u}{2}\right)^2\right\vert_{\b^2\b_c^{-2}\delta}^1 
\leq \frac{r}{2}\left(\log \frac{\b^2\b_c^{-2}\delta}{2}\right)^2 
\leq r (\log \frac{\delta}{2})^2 + r( \log (\b^2\b_c^{-2}))^2.
\eeq

Thus, if $\b$ is exactly the threshold $\b_c$, by taking $\delta=1/N$, we obtain the first part of the theorem. On the other hand, if $\b$ is approaching $\b_c$ as $\b^2=\b_c^2 + d N^{-\a}$ for $d>0$, then by taking $\delta = dN^{-\a}$, we have
\beqq
r (\log \frac{\delta}{2})^2 + r( \log (\b^2\b_c^{-2}))^2 \leq r\left[ (-\a\log N + \log \frac{d}{2})^2 + (\log(1+d\b_c^{-2}N^{-\a}))^2 \right],
\eeqq
as well as 
\beqq
C\b^2 N \left( 1-\b_c^2\b^{-2} + \delta \right) = C\b^2 N \left(  \frac{dN^{-\a}}{\b_c^2+dN^{-\a}}+ dN^{-\a} \right) \leq C\b^2(\b_c^{-2}+1) dN^{1-\a}.
\eeqq
Hence, for sufficiently large $N$, the second part of the theorem follows. 

\end{proof}

Having proved Theorem \ref{thm:maintheorem}, we return to the proof of Lemma \ref{lem:mainlemma}.

\begin{proof}[Proof of Lemma \ref{lem:mainlemma}]
Parts of the proof are very similar to \cite{ChenLam19}, but we include all steps to make this note self-contained.  To bound $\bE \langle \cR(\s,\r) \rangle_t$, we interpolate the measure with respect to another parameter, $\la$. For any $t\in[0,1]$ and $\la\in\RR$, define
\beq
\phi_N(t,\la)=\frac1N\EE\log\sum_{\s,\r}\exp(\b(H^1_{N,t}(\s)+H^2_{N,t}(\r))+\la\b^2N\cR(\s,\r)).
\eeq
The function $\phi_N$ corresponds to a free energy with respect to the measure
\beq
\frac{\exp(\b(H^1_{N,t}(\s)+H^2_{N,t}(\r))+\la\b^2N\cR(\s,\r))}{\sum_{\s,\r}\exp(\b(H^1_{N,t}(\s)+H^2_{N,t}(\r))+\la\b^2N\cR(\s,\r))}
\eeq
and we denote by $\ip{\cdot}_{t,\la}$ the Gibbs expectation with respect to this measure.  We observe that, in terms of the single-parameter interpolation that appears in \eqref{eq:Chatterjee}, we have $\ip{\cdot}_{t,0}=\ip{\cdot}_t$.

The key technique used here is Gaussian integration by parts (see e.g. \cite{Talagrandvol1} Appendix 4).  In particular, given a Gaussian function $f(\s,\r)$, we have
\beq
\EE\ip{f(\s,\r)}_{t,\la}=\EE\ip{\EE f(\s,\r)g(\s,\r)-\EE f(\s,\r)g(\s',\r')}_{t,\la},
\eeq
where $g(\s,\r):=\b(H^{1}_{N,t}(\s)+H^2_{N,t}(\r))$ and $(\s,\r),(\s',\r')$ denote independent copies from the Gibbs measure $\ip{\cdot}_{t,\la}$.  We will also need the following identities: For any $t\in(0,1)$ and any $\s,\s',\r,\r'\in\{-1,1\}^N$,
\beq\begin{split}
&\EE\left(\frac{H_N(\s)}{\sqrt{t}}-\frac{H'_N(\s)}{\sqrt{1-t}}\right)H^1_{N,t}(\s')=0,\qquad
\EE\left(\frac{H_N(\s)}{\sqrt{t}}-\frac{H'_N(\s)}{\sqrt{1-t}}\right)H^2_{N,t}(\r')=N\cR(\s,\r'),\\
&\EE\left(\frac{H_N(\r)}{\sqrt{t}}-\frac{H''_N(\r)}{\sqrt{1-t}}\right)H^2_{N,t}(\r')=0,\qquad
\EE\left(\frac{H_N(\r)}{\sqrt{t}}-\frac{H''_N(\r)}{\sqrt{1-t}}\right)H^2_{N,t}(\s')=N\cR(\s',\r).
\end{split}\eeq
Using Gaussian integration by parts and the identities above, we evaluate the derivative of $\phi_N(t,\la)$ with respect to the first coordinate and obtain
\beq\begin{split}
\partial_1\phi_N(t,\la)&=\frac{\b}{2N}\EE\ipa{\left(\frac{H_N(\s)}{\sqrt{t}}-\frac{H'_N(\s)}{\sqrt{1-t}}\right)+\left(\frac{H_N(\r)}{\sqrt{t}}-\frac{H''_N(\r)}{\sqrt{1-t}}\right)}_{t,\la}\\
&=\frac{\b^2}{2}(2\EE\ip{\cR(\s,\r)}_{t,\la}-\EE\ip{\cR(\s,\r')}_{t,\la}-\EE\ip{\cR(\s',\r)}_{t,\la})\\
&=\b^2(\EE\ip{\cR(\s,\r)}_{t,\la}-\EE\ip{\cR(\s,\r')}_{t,\la}),
\end{split}\eeq
where, in the last step, we used the fact that $(\s,\r')$ and $(\s',\r)$ have the same distribution under the measure $\EE\ip{\cdot}_{t,\la}$.  Next, it will be helpful to consider a shifted version of $\phi(t,\la)$, namely
\beq
\Phi_N(t,\la):=\phi_N(t,\la-t),
\eeq
which has $t$-derivative
\beq\begin{split}
\partial_1\Phi_N(t,\la)&=\partial_1\phi_N(t,\la-t)-\partial_2\phi_N(t,\la-t)\\
&=\b^2\big(\EE\ip{\cR(\s,\r)}_{t,\la-t}-\EE\ip{\cR(\s,\r')}_{t,\la-t}\big)-\b^2\EE\ip{\cR(\s,\r)}_{t,\la-t}\\
&=-\b^2\EE\ip{\cR(\s,\r')}_{t,\la-t}.
\end{split}\eeq
From this, we obtain
\beq
\phi_N(t,\la)=\Phi_N(t,\la+t)
=\int_0^t\partial_1\Phi_N(s,\la+t)\dd s+\Phi_N(0,\la+t).
\eeq
From the derivative computations, we have $\partial_1\Phi_N(s,\la+t)=-\b^2\EE\ip{\cR(\s,\r')}_{s,\la+t-s}$.  Furthermore, the function $\cR(\s,\r')$ is non-negative for all spin vectors $\s,\r'$, due to $\Delta^2$ being positive semi-definite. 
Using this, along with the fact that $\Phi_N(0,\la+t)=\phi_N(0,\la+t)$ we obtain the bound
\beq
\phi_N(t,\la)\leq\phi_N(0,\la+t).
\eeq
Combining this with the convexity of $\phi_n$ in its second input, we have
\beq\begin{split}
\b^2\la\EE\ip{\cR(\s,\r)}_{t,0}&=\la\partial_2\phi_N(t,0)\\
&\leq\phi_N(t,\la)-\phi_N(t,0)\\
&\leq\phi_N(0,\la+t)-\phi_N(t,0)\\
&=\phi_N(0,\la+t)-\phi_N(0,0)\\
&=\frac1N\EE\log\ipa{\exp(\b^2(\la+t)N\cR(\s,\r))}_{0,0}
\end{split}\eeq
Finally, applying Jensen's inequality to last expression above, we obtain
\beq\label{eq:double_exp_bound}
\b^2\la\EE\ip{\cR(\s,\r)}_{t,0}\leq\frac1N\log\EE\ip{\exp(\b^2(\la+t)N\cR(\s,\r)}_{0,0}.
\eeq
Lemma \ref{lem:mainlemma} is now an immediate result of Lemma \ref{lem:Talagrand_bound}, displayed below.  In particular, taking $x=\b^2(\la+t)$ in Lemma \ref{lem:Talagrand_bound}, we choose $\la$ to satisfy the constraint $x<\b_c^2$.  One value that works is 
$\la=\frac12(\b_c^2\b^{-2}-t)$.
Substituting this value of $\la$ into \eqref{eq:double_exp_bound} and applying Lemma \ref{lem:Talagrand_bound}, we obtain Lemma \ref{lem:mainlemma}.
\end{proof}
\begin{lemma}\label{lem:Talagrand_bound}
Consider a $k$-species model with positive semidefinite $\Delta^2$ and let $r=\rank(\Delta^2)$.  Then, for any $x$ satisfying $0<x<\b_c^2$, 
\beq
\EE\ip{\exp(xN\cR(\s,\r))}_{0,0}\leq(1-\b_c^{-2}x)^{-r/2}.
\eeq
\end{lemma}

\begin{proof}[Proof of Lemma \ref{lem:Talagrand_bound}]
Here we adapt a technique used by Talagrand to show an analogous result for the single species model (see \cite{Talagrandvol2}, (A.19)).  Recall 
\beq
\cR(\s,\r)=v^T\Delta^2 v\quad \text{where }v=\begin{bmatrix}R_1\\ \vdots \\ R_k\end{bmatrix}\text{ and }R_s=\frac1N\sum_{i\in I_s}\s_i\r_i.
\eeq
Let $\tilde{v}=\sqrt{2xN}v$. Then we have
\beq
\EE\ip{\exp(xN\cR(\s,\r))}_{0,0}=\EE\ip{\exp(\tfrac12\tilde{v}^T\Delta^2\tilde{v})}_{0,0}=\EE\ip{\EE_g\exp(\tilde{v}^Tg)}_{0,0},\qquad g\sim\cN(0,\Delta^2)
\eeq
where $\EE_g$ denotes expectation with respect to a Gaussian vector $g$ that is independent of $\tilde{v}$. 
We observe that, under the measure $\bE \langle \cdot \rangle_{0,0}$, the variables $\{R_s(\s,\r)\}_{s=1}^k$ are distributed as $\{\frac1N\sum_{i_s\in I_s}X_{i_s}\}_{s=1}^k$ where $X_i$ are iid Rademacher random variables, so this can be evaluated as
\beq\begin{split}
\EE\ip{\EE_g\exp(\tilde{v}^Tg)}_{0,0}
&=\EE_g\prod_{s=1}^k\prod_{i\in I_s}\EE_{X_i}\exp(\sqrt{2xN}\cdot\tfrac1N X_ig_s)\\
&=\EE_g\prod_{s=1}^k\prod_{i\in I_s}\cosh\left(\sqrt{\tfrac{2x}{N}}g_s\right)
=\EE_g\exp\left(\sum_{s=1}^k |I_s|\log\cosh\left(\sqrt{\tfrac{2x}{N}}g_s\right)\right).
\end{split}\eeq
Applying the bound $\log\cosh t \leq t^2/2$, we get
\beq
\EE_g\exp\left(\sum_{s=1}^k|I_s|\log\cosh\left(\sqrt{\tfrac{2x}{N}}g_s\right)\right)
\leq \EE_g\exp\left(x\sum_{s=1}^k \frac{|I_s|}{N} g_s^2\right)=\EE_g\exp(xg^T\Lambda_N g)
\eeq
where $\Lambda_N= \Lambda+O(N^{-1})$ is assumed in the theorem.
For the remainder of the computations, we replace $\Lambda_N$ by $\Lambda$, without precisely tracking the error bound. Now we compute the Gaussian expectation.  We observe that, since $\Delta^2$ is positive semidefinite and symmetric, there exists a $k\times r$ matrix $A$ such that $\Delta^2=AA^T$ and $A^TA$ is positive definite.  Furthermore, $g \stackrel{d}{=} Az$ where $z\sim\cN(0,I_r)$.  Using this, we have
\beq\begin{split}
\EE_g\exp(xg^T\Lambda g)&=\EE_z\exp(xz^TA^T\Lambda Az)\\
&=(2\pi)^{-r/2}\int_{\RR^r}\exp(-\tfrac12u^T(I_r-2xA^T\Lambda A)u)\dd u\\
&=(\det(I_r-2xA^T\Lambda A))^{-1/2}
=x^{-r/2}(\det(\tfrac1x I_r-2A^T\Lambda A))^{-1/2}.
\end{split}\eeq
Thus, we have
\beq
\EE\ip{\exp(xN\cR(\s,\r))}_{0,0} \leq x^{-r/2}(\det(\tfrac1x I_r-2A^T\Lambda A))^{-1/2}.
\eeq
Let $\rho(\cdot)$ denote spectral radius and observe that $\rho(2A^T\Lambda A)=\rho(2\Lambda AA^T)=\rho(2\Lambda\Delta^2)=\b_c^{-2}.$
 Furthermore, $\frac1x>\b_c^{-2}$ by assumption, which implies
 \beq
 \det(\tfrac1x I_r-2A^T\Lambda A)\geq(\tfrac1x-\rho(2A^T\Lambda A))^r=(\tfrac1x-\b_c^{-2})^r
 \eeq
 and the lemma immediately follows.
\end{proof}

\subsection*{Acknowledgements}
The authors wish to thank Jinho Baik, Erik Bates, and Wei-Kuo Chen for helpful discussions.

\newpage


\begin{bibdiv}
\begin{biblist}


\bibitem{AizenmanLebowitzRuelle}
M.~Aizenman, J.~L. Lebowitz, and D.~Ruelle.
\newblock Some rigorous results on the {S}herrington-{K}irkpatrick spin glass
  model.
\newblock {\em Comm. Math. Phys.}, 112(1):3--20, 1987.

\bibitem{Albericietal21}
D.~Alberici, F.~Camilli, P.~Contucci, and E.~Mingione.
\newblock The multi-species mean-field spin-glass on the {N}ishimori line.
\newblock {\em J. Stat. Phys.}, 182(1):Paper No. 2, 20, 2021.

\bibitem{conjSk2}
T.~Aspelmeier.
\newblock Free-energy fluctuations and chaos in the {S}herrington-{K}irkpatrick
  model.
\newblock {\em Phys. Rev. Lett.}, 100:117205, Mar 2008.

\bibitem{conjSK1}
T.~Aspelmeier, A.~Billoire, E.~Marinari, and M.~A. Moore.
\newblock Finite-size corrections in the {S}herrington-{K}irkpatrick model.
\newblock {\em J. Phys. A}, 41(32):324008, 21, 2008.

\bibitem{Barra15}
A.~Barra, P.~Contucci, E.~Mingione, and D.~Tantari.
\newblock Multi-species mean field spin glasses. {R}igorous results.
\newblock {\em Ann. Henri Poincar\'e}, 16(3):691--708, 2015.

\bibitem{BatesSolmanSohn19}
E.~Bates, L.~Sloman, and Y.~Sohn.
\newblock Replica symmetry breaking in multi-species
  {S}herrington-{K}irkpatrick model.
\newblock {\em J. Stat. Phys.}, 174(2):333--350, 2019.

\bibitem{Chatterjee09}
S.~Chatterjee.
\newblock Disorder chaos and multiple valleys in spin glasses, 2009.

\bibitem{ChenLam19}
W.-K. Chen and W.-K. Lam.
\newblock Order of fluctuations of the free energy in the {SK} model at
  critical temperature.
\newblock {\em ALEA Lat. Am. J. Probab. Math. Stat.}, 16(1):809--816, 2019.

\bibitem{DeyWu21}
P.~S. Dey and Q.~Wu.
\newblock Fluctuation results for multi-species {S}herrington-{K}irkpatrick
  model in the replica symmetric regime.
\newblock {\em J. Stat. Phys.}, 185(3):Paper No. 22, 40, 2021.

\bibitem{DeyWu23}
P.~S. Dey and Q.~Wu.
\newblock Mean field spin glass models under weak external field.
\newblock {\em Comm. Math. Phys.}, 402(2):1205--1258, 2023.

\bibitem{Guerra03}
F.~Guerra.
\newblock Broken replica symmetry bounds in the mean field spin glass model.
\newblock {\em Communications in Mathematical Physics}, 233(1):1--12, 2003.

\bibitem{Panchenk14}
D.~Panchenko.
\newblock The {P}arisi formula for mixed {$p$}-spin models.
\newblock {\em Ann. Probab.}, 42(3):946--958, 2014.

\bibitem{Panchenko15}
D.~Panchenko.
\newblock The free energy in a multi-species {S}herrington-{K}irkpatrick model.
\newblock {\em Ann. Probab.}, 43(6):3494--3513, 2015.

\bibitem{Parisi1979}
G.~Parisi.
\newblock Infinite number of order parameters for spin-glasses.
\newblock {\em Phys. Rev. Lett.}, 43:1754--1756, Dec 1979.

\bibitem{Parisi1980}
G.~Parisi.
\newblock The order parameter for spin glasses: a function on the interval 0-1.
\newblock {\em J. Phys. A: Math. Gen.}, 13(3):1101--1112, 1980.

\bibitem{ParisiRizzo09}
G.~Parisi and T.~Rizzo.
\newblock Phase diagram and large deviations in the free energy of mean-field
  spin glasses.
\newblock {\em Phys. Rev. B}, 79:134205, Apr 2009.

\bibitem{SherringtonKirkpatrick1975}
D.~Sherrington and S.~Kirkpatrick.
\newblock Solvable model of a spin-glass.
\newblock {\em Phys. Rev. Lett.}, 35(26):1792--1796, 1975.

\bibitem{TalagrandSK}
M.~Talagrand.
\newblock The {P}arisi formula.
\newblock {\em Ann. of Math. (2)}, 163(1):221--263, 2006.

\bibitem{Talagrandvol1}
M.~Talagrand.
\newblock {\em Mean field models for spin glasses. {V}olume {I}}, volume~54 of
  {\em Results in Mathematics and Related Areas. 3rd Series. A Series of Modern
  Surveys in Mathematics}.
\newblock Springer-Verlag, Berlin, 2011.
\newblock Basic examples.

\bibitem{Talagrandvol2}
M.~Talagrand.
\newblock {\em Mean field models for spin glasses. {V}olume {II}}, volume~55 of
  {\em Results in Mathematics and Related Areas. 3rd Series. A Series of Modern
  Surveys in Mathematics}.
\newblock Springer, Heidelberg, 2011.
\newblock Advanced replica-symmetry and low temperature.

\bibitem{Wu24}
Q.~Wu.
\newblock Thouless-{A}nderson-{P}almer equations for the multi-species
  {S}herrington-{K}irkpatrick model.
\newblock {\em J. Stat. Phys.}, 191(7):Paper No. 87, 14, 2024.


\end{biblist}
\end{bibdiv}

\end{document}